\documentclass[11pt,a4paper]{article}
\usepackage[T1]{fontenc}
\usepackage{lmodern}
\usepackage[margin=2.7cm]{geometry}
\usepackage{amsmath,amsthm}
\usepackage{tikz}
\usepackage[hidelinks]{hyperref}
\hypersetup{pdftitle={Four colours are sometimes needed for closed-neighbourhood
conflict-free colouring of planar graphs},pdfauthor={Jelena Sedlar, Riste Skrekovski}}
\definecolor{gA}{RGB}{0,90,170}
\definecolor{gB}{RGB}{0,130,60}
\definecolor{gC}{RGB}{200,100,0}
\definecolor{gD}{RGB}{0,150,170}
\definecolor{gE}{RGB}{130,40,160}
\newtheorem{theorem}{Theorem}
\newtheorem{lemma}[theorem]{Lemma}
\newtheorem{corollary}[theorem]{Corollary}
\newtheorem{proposition}[theorem]{Proposition}
\newtheorem{conjecture}[theorem]{Conjecture}
\theoremstyle{definition}
\newtheorem{definition}[theorem]{Definition}
\theoremstyle{remark}

\newcommand{\cf}{\chi_{\mathrm{iCFc}}}
\newcommand{\W}{\mathcal{W}}
\begin{document}

\begin{center}
{\Large\bfseries Four colours are  needed for closed-neighbourhood\\[0.1cm]
conflict-free colouring of all planar graphs}\\[0.4cm]
Jelena Sedlar$^{1,3}$, Riste \v{S}krekovski$^{2,3,4}$\\[0.2cm]
{\small $^{1}$ \textit{University of Split, Faculty of Civil Engineering, Architecture and
Geodesy, Split, Croatia}}\\[0.05cm]
{\small $^{2}$ \textit{University of Ljubljana, FMF, Ljubljana, Slovenia}}\\[0.05cm]
{\small $^{3}$ \textit{University of Novo Mesto, FIS, Novo Mesto, Slovenia}}\\[0.05cm]
{\small $^{4}$ \textit{Rudolfovo -- Science and Technology Centre, Novo Mesto,
Slovenia}}
\end{center}

\begin{abstract}
A colouring of the vertices of a graph is conflict-free with respect to closed
neighbourhoods if every closed neighbourhood contains a colour that appears on
exactly one of its vertices. The colouring need not be proper. Every planar graph
has such a colouring with four colours. Fabrici, Lu\v{z}ar, Rindo\v{s}ov\'a and
Sot\'ak conjectured that three colours always suffice. We disprove this
conjecture. We give a simple gadget construction that turns every planar graph of
chromatic number four into a planar graph with no such colouring with three
colours. This also shows that it is NP-complete to decide whether three colours
suffice for a planar graph. We also give a planar graph on $25$ vertices that
needs four colours. Further, we give
critical planar graphs on $25$ vertices with $52$ and with $51$ edges, an infinite
family of plane triangulations that need four colours, and an infinite family of
critical examples.
\end{abstract}


\section{Introduction}\label{SecIntro}

Conflict-free colourings of hypergraphs were introduced by Even, Lotker, Ron
and Smo\-ro\-din\-sky \cite{Even}. For a survey see \cite{Smorodinsky}. A colouring
of the vertices of a graph $G$ is \emph{conflict-free with respect to closed
neighbourhoods}, or a \emph{CF-colouring} for short, if every closed
neighbourhood $N[v]$ contains a colour that appears on exactly one of its
vertices. The colouring need not be proper. Equivalently, it is a conflict-free
colouring of the hypergraph $\{N[v]:v\in V(G)\}$. The minimum number of colours
in a CF-colouring of $G$ is denoted by $\cf(G)$ \cite{Fabrici}. The subscript
stands for \emph{improper conflict-free} colouring with respect to \emph{closed}
neighbourhoods. In a proper colouring the colour of $v$ appears exactly once in
$N[v]$. Hence $\cf(G)\le\chi(G)$, and by the Four Colour Theorem $\cf(G)\le4$ for
every planar graph $G$.

Pach and Tardos \cite{PachTardos} studied this parameter, and bounds in terms of
the maximum degree were given in \cite{PachTardos,GST,BKM}. Abel et
al.\ \cite{Abel} studied the variant in which some vertices may stay uncoloured.
They proved that in this variant three colours suffice for every planar graph,
which again gives $\cf(G)\le4$. For the class $\mathcal P$ of planar graphs let
$\cf(\mathcal P)=\max\{\cf(G):G\in\mathcal P\}$. Fabrici, Lu\v{z}ar,
Rindo\v{s}ov\'a and Sot\'ak \cite{Fabrici} observed that $3\le\cf(\mathcal P)\le4$.
The lower bound is attained, for example, by the $5$-cycle with one chord. They
proposed the following conjecture, which we state in an equivalent form.

\begin{conjecture}[{\cite[Conjecture~9.1(b)]{Fabrici}}]\label{Con}
Every planar graph $G$ satisfies $\cf(G)\le3$.
\end{conjecture}

For the analogous notion with respect to open neighbourhoods, Hlin\v{e}n\'y
and M\'alik \cite{HlinenyMalik} recently proved that four colours suffice for
planar graphs and that this bound is tight. In contrast, we show that
Conjecture~\ref{Con} is false. Thus four colours are needed for closed
neighbourhoods as well. Our main result is the following.

\begin{theorem}\label{Main}
There is a planar graph $G$ on $25$ vertices with $\cf(G)=4$. Moreover, deciding
whether $\cf(G)\le3$ is NP-complete for planar graphs.
\end{theorem}

The two parts of Theorem~\ref{Main} are proved in Section~\ref{SecG25}
(Theorem~\ref{T25}) and Section~\ref{SecGadget} (Corollary~\ref{NP}).

All graphs are finite and simple. For a vertex $v$ of a graph $G$, we denote by
$N_G(v)$ its open neighbourhood and by $N_G[v]=N_G(v)\cup\{v\}$ its closed
neighbourhood. We omit the subscript when $G$ is clear. A \emph{colouring} of $G$
is any map from $V(G)$ to a set of colours. It need not be proper. For a colouring
$c$ and a set $S$ of vertices we write $c(S)=\{c(v):v\in S\}$. A colour is
\emph{unique} in $S$ if it appears on exactly one vertex of $S$. The set $S$ is
\emph{good} if some colour is unique in $S$. Thus $c$ is a CF-colouring if and
only if $N[v]$ is good for every vertex $v$. A \emph{$3$-CF-colouring} is a
CF-colouring with colours $1,2,3$. In contrast, a \emph{$3$-colouring} is always
proper. We often use the following fact. If every vertex of a set $S$ has colour
$a$ or $b$, then $S$ is good if and only if $a$ appears on exactly $1$ or exactly
$|S|-1$ vertices of $S$. In particular, a set of three vertices is good if and
only if it is not monochromatic.

A planar graph $G$ with $\cf(G)=4$ is \emph{critical} if $\cf(G-v)\le3$ and
$\cf(G-e)\le3$ for every vertex $v$ and every edge $e$ of $G$. Every planar graph $G$ with $\cf(G)=4$ contains a
critical subgraph. Indeed, a subgraph $G'$ of $G$ with $\cf(G')=4$ that is minimal
under inclusion is critical, since every $G'-v$ and $G'-e$ is a proper subgraph
of $G'$. On the other hand, a critical graph need not be inclusion-minimal with
$\cf=4$. Deleting two or more elements may again give a graph with $\cf=4$.

In Section~\ref{SecGadget} we introduce forcing gadgets. In
Section~\ref{SecG25} we study the graph on $25$ vertices and give two critical
graphs on $25$ vertices. In Section~\ref{SecFam} we give an infinite family of
plane triangulations and an infinite family of critical graphs with $\cf=4$.

\section{Forcing gadgets}\label{SecGadget}

For vertices $x,y$ let $W(x,y)$ be the graph that consists of a $4$-cycle
$r_1r_2r_3r_4$, the vertex $x$ joined to all four vertices of the cycle, and the
vertex $y$ joined only to $r_4$. We call the cycle the \emph{rim} and $x$ the
\emph{apex} (see Figure~\ref{FigW}). The rim vertices are the \emph{inner}
vertices of $W(x,y)$ in the sense of the following definition.

\begin{definition}\label{D_force}
Let $Q$ be a graph with two distinct vertices $x$ and $y$, and let
$I\subseteq V(Q)\setminus\{x,y\}$. The vertices in $I$ are called \emph{inner}.
We call $Q$ a \emph{forcing gadget with apex $x$} if the following holds for
every colouring $c$ of $Q$ with colours $1,2,3$ such that $c(x)=c(y)$ and
$N_Q[u]$ is good for every $u\in I$. Both colours different from $c(x)$ appear
on $N_Q(x)\cap I$. A copy of $Q$ is \emph{attached} to an edge $xy$ of a graph
$G$ if $Q\subseteq G$, the vertices $x,y$ of $Q$ are the ends of the edge $xy$ of
$G$, and $N_G[u]=N_Q[u]$ for every $u\in I$.
\end{definition}

\begin{figure}[t]
\centering
\begin{tikzpicture}[scale=0.8,v/.style={circle,draw,fill=white,inner sep=1.2pt}]
\node[v,label=left:$x$] (x) at (0,0) {};
\node[v,label=right:$y$] (y) at (6,0) {};
\draw[thick] (x)--(y);
\foreach \s/\p in {1/r,-1/s}{
\node[v] (\p1) at (1.4,2.3*\s) {};
\node[v] (\p2) at (2.2,1.9*\s) {};
\node[v] (\p3) at (2.8,1.3*\s) {};
\node[v] (\p4) at (3.3,0.6*\s) {};
\draw (\p1)--(\p2)--(\p3)--(\p4);
\draw (\p1) .. controls (3.6,3.2*\s) and (4.6,1.6*\s) .. (\p4);
\foreach \r in {\p1,\p2,\p3,\p4} \draw (x)--(\r);
\draw (y)--(\p4);}
\node[above left=-1pt] at (r1) {\small$r_1$};
\node[above right=-1pt] at (r2) {\small$r_2$};
\node[above right=-1pt] at (r3) {\small$r_3$};
\node[below right=0pt] at (r4) {\small$r_4$};
\node[below left=-1pt] at (s1) {\small$r'_1$};
\node[below right=-1pt] at (s2) {\small$r'_2$};
\node[below right=-1pt] at (s3) {\small$r'_3$};
\node[above right=0pt] at (s4) {\small$r'_4$};
\end{tikzpicture}
\caption{Two copies of $W(x,y)$, with rims $r_1r_2r_3r_4$ and $r'_1r'_2r'_3r'_4$,
attached to an edge $xy$.}
\label{FigW}
\end{figure}
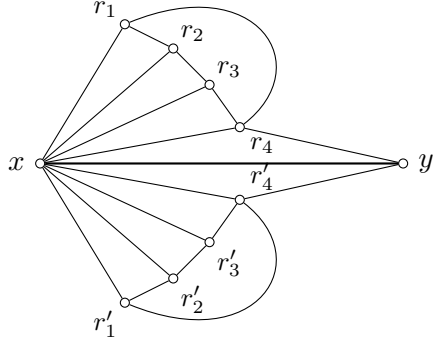

\begin{lemma}\label{L_rim}
$W(x,y)$ is a forcing gadget with apex $x$.
\end{lemma}

\begin{proof}
Let $c$ be a colouring of $W(x,y)$ with colours $1,2,3$ such that $c(x)=c(y)=a$
and $N[r_j]$ is good for $j=1,\dots,4$. Suppose that some colour $d\neq a$ does
not appear on the rim. Then the rim is coloured only by $a$ and the remaining
colour $b$. For $j=1,\dots,4$ let $\sigma_j$ be the number of vertices of colour
$a$ in $N[r_j]\cap\{r_1,\dots,r_4\}$. This is the rim without $r_{j+2}$ (indices
modulo $4$). For $j\le3$, the set $N[r_j]$ consists of $x$ and these three rim
vertices. So it is good if and only if $1+\sigma_j\in\{1,3\}$, that is,
$\sigma_j\in\{0,2\}$. The set $N[r_4]$ consists of $x$, $y$ and these three rim
vertices. So it is good if and only if $2+\sigma_4=4$. Hence all $\sigma_j$ are
even. Let $\sigma$ be the number of vertices of colour $a$ on the rim. Then
$\sigma_j=\sigma-1$ if $c(r_{j+2})=a$, and $\sigma_j=\sigma$ otherwise. Since all
$\sigma_j$ have the same parity, all rim vertices have the same colour. Then
$\sigma_4\in\{0,3\}$, a contradiction.
\end{proof}

We use two kinds of colourings of the rim of a copy of $W(x,y)$ whose rim
vertices have no further neighbours. Both are easy to check.
\begin{description}
\item[(W1)] Suppose that no rim vertex has colour $c(x)$, and that $y$ is absent
or $c(y)\ne c(x)$. Then $c(x)$ is unique in $N[r]$ for every rim vertex $r$. If
$c(x)\ne c(y)$ and all rim vertices get the third colour, we call the colouring
of $W(x,y)$ \emph{standard}. Then $x$ and $y$ see only the third colour on the
rim.
\item[(W2)] Suppose that $c(x)=c(y)=a$, one rim vertex $r_j$ with $j\ne2$ gets a
colour $d\ne a$, and the other three rim vertices get the third colour. Then $a$
or $d$ is unique in $N[r]$ for every rim vertex $r$. Indeed, for $r\ne r_4$ the
colour $a$ appears in $N[r]$ only on $x$, and $N[r_4]$ contains $r_j$.
\end{description}

\begin{lemma}[Forcing lemma]\label{L_force}
Let $Q$ be a forcing gadget with apex $x$ and set of inner vertices $I$, and let
$Q$ be attached to an edge $xy$ of a graph $G$. Let $c$ be a colouring of $G$
with colours $1,2,3$ such that $c(x)=c(y)$ and the sets $N[x]$ and $N[u]$,
$u\in I$, are good. Then some colour different from $c(x)$ does not appear on
$N(x)\setminus I$. In particular, suppose that two forcing gadgets with apex $x$
and disjoint sets of inner vertices are attached to $xy$. Then every
$3$-CF-colouring $c$ of $G$ satisfies $c(x)\ne c(y)$.
\end{lemma}

\begin{proof}
Let $a=c(x)$. Since $N_G[u]=N_Q[u]$ for $u\in I$, both colours different from
$a$ appear on $N(x)\cap I$. The colour $a$ appears on $x$ and $y$, and $N[x]$ is
good. So some colour $b\ne a$ is unique in $N[x]$, and it appears in $I$. Hence
$b$ does not appear on $N(x)\setminus I$. For the second statement, apply this to
one of the gadgets. The set $N(x)\setminus I$ contains the inner neighbours of
$x$ in the other gadget, and both colours different from $a$ appear on them.
\end{proof}

\begin{theorem}\label{Tm}
Let $G_0$ be a planar graph. For every edge of $G_0$, choose one of its ends as
the apex $x$ and the other as $y$, and attach two copies of $W(x,y)$ whose rims
consist of new vertices. Let $G$ be the resulting graph. Then $G$ is planar, and
$G$ has a $3$-CF-colouring if and only if $G_0$ is $3$-colourable. In particular,
$\cf(G)=4$ if $\chi(G_0)=4$.
\end{theorem}

\begin{proof}
Both copies can be drawn close to the edge $xy$ (see Figure~\ref{FigW}), so $G$
is planar. Both copies are attached to $xy$. So by Lemmas~\ref{L_rim}
and~\ref{L_force} every $3$-CF-colouring of $G$ is a proper colouring of $G_0$.
Conversely, colour $G_0$ properly with three colours, and colour every copy of
$W(x,y)$ standardly. The rim vertices are then fine by (W1). The colour of every
vertex $v$ of $G_0$ is unique in $N[v]$, since the neighbours of $v$ in $G_0$ and
the rim vertices adjacent to $v$ have other colours. Finally, $\cf(G)\le4$ by the
Four Colour Theorem.
\end{proof}

Both copies are needed, and they must have the same apex. For $G_0=K_4$, the
graph has a $3$-CF-colouring if every edge $xy$ carries only one copy, or one copy
of $W(x,y)$ and one copy of $W(y,x)$.

\begin{corollary}\label{NP}
Deciding whether $\cf(G)\le3$ is NP-complete for planar graphs, even for planar
graphs of maximum degree at most $24$.
\end{corollary}

\begin{proof}
The problem is in NP. It is NP-complete to decide whether a planar graph $G_0$
of maximum degree at most $4$ is $3$-colourable \cite{GJS}. Orient the edges of
$G_0$ so that the in-degree and the out-degree of every vertex differ by at most
one. For example, join a new vertex to all vertices of odd degree and orient the
edges along an Euler tour. Let the tail of every edge be its apex. The graph $G$
of Theorem~\ref{Tm} has $|V(G_0)|+8|E(G_0)|$ vertices. Every vertex of $G_0$ is
the apex of at most two edges, so it has degree at most $4+2\cdot8+2\cdot2=24$ in
$G$. Rim vertices have degree at most $4$. By Theorem~\ref{Tm}, $\cf(G)\le3$ if
and only if $G_0$ is $3$-colourable.
\end{proof}

The graphs of Theorem~\ref{Tm} are far from critical.

\begin{proposition}\label{P_noncrit}
Let $G$ be a graph of Theorem~\ref{Tm} with $\chi(G_0)=4$, and let $F$ be a
nonempty set of rim vertices and edges of one of the attached copies of $W(x,y)$.
Then $\cf(G-F)=4$. In particular, $G$ is not critical.
\end{proposition}

\begin{proof}
The graph $G-F$ is planar, so $\cf(G-F)\le4$. Suppose that $G-F$ has a
$3$-CF-colouring $c$. Let $F$ lie in a copy of $W(x,y)$ attached to the edge $xy$
of $G_0$, and let $W'$ be the other copy attached to $xy$. Let $x'y'\ne xy$ be an
edge of $G_0$. Both copies of $W(x',y')$ are attached to $x'y'$ in $G-F$. So
$c(x')\ne c(y')$ by Lemmas~\ref{L_rim} and~\ref{L_force}. Since $G_0$ has no
proper $3$-colouring, $c(x)=c(y)=a$. Every neighbour of $x$ in $G_0$ other than
$y$ has a colour different from $a$. Suppose that $x$ has no neighbour of colour
$b$ in $G_0$, for some colour $b\ne a$. Then recolouring $x$ with $b$ gives a
proper $3$-colouring of $G_0$, since $c(y)=a\ne b$. Hence both colours different
from $a$ appear on $N_{G_0}(x)\setminus\{y\}$. This set is contained in
$N_{G-F}(x)$ and is disjoint from the rim of $W'$. This contradicts
Lemma~\ref{L_force} applied to $W'$, which is attached to $xy$ in $G-F$.
\end{proof}

\section{A counterexample on 25 vertices}\label{SecG25}

Let $G_{25}$ be the planar graph on $25$ vertices and $55$ edges shown in
Figure~\ref{FigG25}. It consists of a $K_4$ on the \emph{hubs} $0,1,2,3$ and five
gadgets. We denote their vertex sets by
\[
A=\{4,\dots,8\},\quad B=\{9,\dots,13\},\quad C=\{14,\dots,17\},\quad
D=\{18,\dots,22\},\quad E=\{23,24\}.
\]
Table~\ref{TabN} lists the neighbourhoods of the gadget vertices. The gadget $C$
is a copy of $W(1,3)$ with $(r_1,r_2,r_3,r_4)=(14,16,15,17)$. Let $G'_{25}$ be
obtained from $G_{25}$ by deleting the three edges $0\,18$, $0\,20$ and $1\,8$.
These edges are dashed in Figure~\ref{FigG25}. Let $G''_{25}$ be obtained from
$G'_{25}$ by replacing the gadget $A$ by the gadget in the last row of
Table~\ref{TabN}. The graphs $G'_{25}$ and $G''_{25}$ have $25$ vertices, and they
have $52$ and $51$ edges, respectively. A straight-line plane drawing of
$G''_{25}$ is obtained from Figure~\ref{FigG25} (left) by replacing $A$ with the
drawing on the right and deleting the dashed edges.

\begin{figure}[t]
\centering
\begin{minipage}[b]{0.52\textwidth}\centering
\begin{tikzpicture}[scale=0.48,
  hub/.style={circle,draw=black,fill=red!80!black,text=white,inner sep=0pt,minimum size=11pt,font=\scriptsize\bfseries},
  gv/.style={circle,draw=black,fill=white,inner sep=0pt,minimum size=8.5pt,font=\fontsize{5}{5}\selectfont}]
\coordinate (p0) at (0,0);
\coordinate (p1) at (13,0);
\coordinate (p2) at (8,12);
\coordinate (p3) at (4.5,5.6);
\coordinate (p4) at (4,2.1);
\coordinate (p5) at (9,0.94);
\coordinate (p6) at (6,0.9);
\coordinate (p7) at (6,4.0);
\coordinate (p8) at (6,2.4);
\coordinate (p9) at (3.0,6.6);
\coordinate (p10) at (4.0,9.0);
\coordinate (p11) at (3.2,10.3);
\coordinate (p12) at (2.9,5.2);
\coordinate (p13) at (0.6,7.8);
\coordinate (p14) at (9.8,5.4);
\coordinate (p15) at (8.3,4.6);
\coordinate (p16) at (10.0,3.6);
\coordinate (p17) at (6.9,5.3);
\coordinate (p18) at (6,-4.0);
\coordinate (p19) at (9,-0.94);
\coordinate (p20) at (6,-0.9);
\coordinate (p21) at (4,-2.1);
\coordinate (p22) at (6,-2.4);
\coordinate (p23) at (9.0,8.3);
\coordinate (p24) at (6.6,8.6);
\draw[gA,line width=0.55pt] (p0)--(p4) (p0)--(p6) (p0)--(p7) (p0)--(p8) (p1)--(p5) (p1)--(p6) (p1)--(p7) (p4)--(p7) (p4)--(p8) (p5)--(p8) (p6)--(p8);
\draw[gA,line width=0.55pt,dashed] (p1)--(p8);
\draw[gB,line width=0.55pt] (p0)--(p9) (p0)--(p12) (p0)--(p13) (p2)--(p9) (p2)--(p10) (p2)--(p11) (p2)--(p13) (p9)--(p12) (p9)--(p13) (p10)--(p13) (p11)--(p13);
\draw[gC,line width=0.55pt] (p1)--(p14) (p1)--(p15) (p1)--(p16) (p1)--(p17) (p3)--(p17) (p14)--(p16) (p14)--(p17) (p15)--(p16) (p15)--(p17);
\draw[gD,line width=0.55pt] (p0)--(p21) (p0)--(p22) (p1)--(p18) (p1)--(p19) (p1)--(p20) (p1)--(p22) (p18)--(p21) (p18)--(p22) (p19)--(p22) (p20)--(p22);
\draw[gD,line width=0.55pt,dashed] (p0)--(p18) (p0)--(p20);
\draw[gE,line width=0.55pt] (p1)--(p23) (p2)--(p23) (p2)--(p24) (p3)--(p24) (p23)--(p24);
\draw[line width=1.3pt] (p0)--(p1) (p0)--(p2) (p0)--(p3) (p1)--(p2) (p1)--(p3) (p2)--(p3);
\node[hub] at (p0) {0};
\node[hub] at (p1) {1};
\node[hub] at (p2) {2};
\node[hub] at (p3) {3};
\node[gv] at (p4) {4};
\node[gv] at (p5) {5};
\node[gv] at (p6) {6};
\node[gv] at (p7) {7};
\node[gv] at (p8) {8};
\node[gv] at (p9) {9};
\node[gv] at (p10) {10};
\node[gv] at (p11) {11};
\node[gv] at (p12) {12};
\node[gv] at (p13) {13};
\node[gv] at (p14) {14};
\node[gv] at (p15) {15};
\node[gv] at (p16) {16};
\node[gv] at (p17) {17};
\node[gv] at (p18) {18};
\node[gv] at (p19) {19};
\node[gv] at (p20) {20};
\node[gv] at (p21) {21};
\node[gv] at (p22) {22};
\node[gv] at (p23) {23};
\node[gv] at (p24) {24};
\node[gA,font=\small\bfseries] at (7.9,2.4) {A};
\node[gD,font=\small\bfseries] at (7.9,-2.4) {D};
\node[gB,font=\small\bfseries] at (1.2,9.6) {B};
\node[gC,font=\small\bfseries] at (8.4,6.4) {C};
\node[gE,font=\small\bfseries] at (7.8,9.4) {E};
\end{tikzpicture}
\end{minipage}\hfill
\begin{minipage}[b]{0.44\textwidth}\centering
\begin{tikzpicture}[scale=0.48,
  hub/.style={circle,draw=black,fill=red!80!black,text=white,inner sep=0pt,minimum size=11pt,font=\scriptsize\bfseries},
  gv/.style={circle,draw=black,fill=white,inner sep=0pt,minimum size=8.5pt,font=\fontsize{5}{5}\selectfont}]
\coordinate (p0) at (0,0); \coordinate (p1) at (13,0);
\coordinate (p4) at (6,4.0); \coordinate (p5) at (9,1.6); \coordinate (p6) at (9,0.9);
\coordinate (p7) at (4,3.3); \coordinate (p8) at (6,2.4);
\draw[gA,line width=0.55pt] (p0)--(p4) (p0)--(p7) (p0)--(p8) (p1)--(p4) (p1)--(p5)
  (p1)--(p6) (p4)--(p7) (p4)--(p8) (p5)--(p8) (p6)--(p8);
\draw[line width=1.3pt] (p0)--(p1);
\node[hub] at (p0) {0}; \node[hub] at (p1) {1};
\foreach \i in {4,...,8} \node[gv] at (p\i) {\i};
\node[gA,font=\small\bfseries] at (7.9,3.2) {A};
\end{tikzpicture}
\\[1.2cm]\end{minipage}
\caption{Left: the planar graph $G_{25}$ with $\cf(G_{25})=4$. The hubs $0,1,2,3$
form the thick $K_4$. The gadgets $A$ and $D$ are on the hubs $0,1$, the gadget
$B$ is on the hubs $0,2$, and the gadgets $C=W(1,3)$ and $E$ are on the hubs
$1,2,3$. Deleting the three dashed edges gives $G'_{25}$. Right: the gadget $A$ of $G''_{25}$, with the hubs $0,1$ at the same
positions.}
\label{FigG25}
\end{figure}
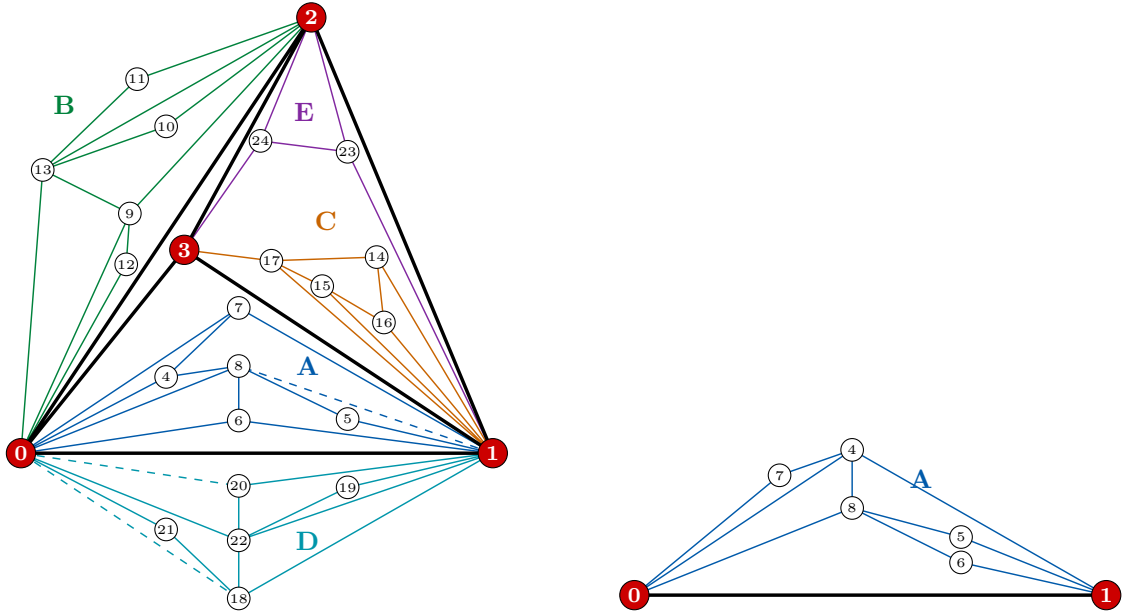

\begin{table}[htb]
\centering\small
\begin{tabular}{@{}l@{\quad}l@{}}
\hline
gadget & vertex $v$: $N(v)$\\
\hline
$A$ & $4{:}\ 0,7,8$;\quad $5{:}\ 1,8$;\quad $6{:}\ 0,1,8$;\quad $7{:}\ 0,1,4$;\quad $8{:}\ 0,(1),4,5,6$\\
$B$ & $9{:}\ 0,2,12,13$;\quad $10{:}\ 2,13$;\quad $11{:}\ 2,13$;\quad $12{:}\ 0,9$;\quad $13{:}\ 0,2,9,10,11$\\
$C$ & $14{:}\ 1,16,17$;\quad $15{:}\ 1,16,17$;\quad $16{:}\ 1,14,15$;\quad $17{:}\ 1,3,14,15$\\
$D$ & $18{:}\ (0),1,21,22$;\quad $19{:}\ 1,22$;\quad $20{:}\ (0),1,22$;\quad $21{:}\ 0,18$;\quad $22{:}\ 0,1,18,19,20$\\
$E$ & $23{:}\ 1,2,24$;\quad $24{:}\ 2,3,23$\\
\hline
$A$ in $G''_{25}$ & $4{:}\ 0,1,7,8$;\quad $5{:}\ 1,8$;\quad $6{:}\ 1,8$;\quad $7{:}\ 0,4$;\quad $8{:}\ 0,4,5,6$\\
\hline
\end{tabular}
\caption{The neighbourhoods of the gadget vertices of $G_{25}$. The edges to the
hubs in parentheses are the three edges deleted in $G'_{25}$. The last row gives
the gadget $A$ of $G''_{25}$.}
\label{TabN}
\end{table}

\begin{theorem}\label{T25}
Let $G$ be obtained from $G_{25}$ by deleting any subset of the three edges
$0\,18$, $0\,20$, $1\,8$. Then $G$ is planar and $\cf(G)=4$. In particular,
$\cf(G_{25})=\cf(G'_{25})=4$.
\end{theorem}

We prove Theorem~\ref{T25} after some setup and three lemmas.

\paragraph{Setup.} Figure~\ref{FigG25} is a straight-line plane drawing of
$G_{25}$. So every such $G$ is planar and $\cf(G)\le4$. For the lower bound, fix
such a graph $G$, let $c$ be a $3$-CF-colouring of $G$, and write $c_i=c(i)$. For
a gadget $X$ and a hub $i$ let $X(i)=N(i)\cap X$. Then $C(1)=C$, $23\in N(1)$ and
$23,24\in N(2)$. For a hub $i$, the set $N[i]$ is the union of $\{0,1,2,3\}$ and
the sets $X(i)$. For a hub $i$ and a colour $\ell$ let $\mu_i(\ell)$ be the number
of vertices of colour $\ell$ in $N[i]$. Note that the multiset of colours on
$\{0,1,2,3\}\subseteq N[i]$ is the same for all hubs. In
Lemmas~\ref{LD}--\ref{LE}, $G$, $c$ and the notation are as in this setup. These
lemmas use only the closed neighbourhoods of the gadget vertices, and these
consist of gadget vertices and hubs.

\paragraph{$D$-gadgets.} Let $u,v$ be hubs. A \emph{$D$-gadget on $(u,v)$} is a
gadget $X$ with five vertices $s,t_1,t_2,w,z$ such that
\[
N[z]=\{u,v,s,t_1,t_2,z\},\quad N[t_1]=\{v,t_1,z\},\quad N[w]=\{u,s,w\},
\]
$N[s]\setminus\{u\}=\{v,s,w,z\}$ and $N[t_2]\setminus\{u\}=\{v,t_2,z\}$. The hub
$u$ may or may not be adjacent to $s$ and to $t_2$. Thus $X(v)=X\setminus\{w\}$
and $w,z\in X(u)$, so $X=X(u)\cup X(v)$. In all eight graphs, $D$ is a $D$-gadget
on $(0,1)$ with $(s,t_1,t_2,w,z)=(18,19,20,21,22)$. Also, $B$ is a $D$-gadget on
$(0,2)$ with $(s,t_1,t_2,w,z)=(9,10,11,12,13)$, where $0$ is adjacent to $9$ but
not to $11$. So $B(0)=\{9,12,13\}$ and $B(2)=\{9,10,11,13\}$.

\begin{lemma}\label{LD}
Let $X$ be a $D$-gadget on $(u,v)$.

\textup{(a)} The set $X\cup\{u,v\}$ receives all three colours.

\textup{(b)} If a colour $d\notin\{c_u,c_v\}$ is missing on $X(v)$, then $c(w)=d$
and $c(z)$ is the colour different from $c_v$ and $d$. In particular, $c(X(u))$
contains both colours different from $c_v$.
\end{lemma}

\begin{proof}
(a) Suppose that only two colours appear on $X\cup\{u,v\}$. First let
$c_u=c_v=a$, and let $b$ be the other colour. By $N[z]$, exactly one of
$s,t_1,t_2,z$ has colour $b$. If $c(z)=a$, then $N[t_1]$ and $N[t_2]$ force
$c(t_1)=c(t_2)=b$, which is impossible. So $c(z)=b$ and $c(s)=a$. Then $N[s]$
forces $c(w)=a$, and $N[w]$ is monochromatic. Now let $c_u\ne c_v$. Then $N[z]$
is coloured by $c_u$ and $c_v$, and each of them appears on a hub. So
$s,t_1,t_2,z$ have a common colour, and this colour is $c_u$ by $N[t_1]$. Then
$N[w]$ gives $c(w)=c_v$. So $N[s]$ contains $c_u$ on $s,z$ and $c_v$ on $v,w$,
and it is not good.

(b) By (a), $d$ appears on $X\setminus X(v)=\{w\}$. Let $p$ be the colour
different from $c_v$ and $d$. So $X(v)$ is coloured by $c_v$ and $p$. Suppose
that $c(z)=c_v$. Then $N[t_1]$ forces $c(t_1)=p$. Since $c_u\ne d$, either
$c_u=p$ or $c_u=c_v$. If $c_u=p$, then $N[z]$ contains $p$ on $u,t_1$ and $c_v$
on $v,z$. If $c_u=c_v$, then $N[t_2]$ forces $c(t_2)=p$, and $N[z]$ contains $p$
on $t_1,t_2$ and $c_v$ on $u,v,z$. In both cases $N[z]$ is not good. Hence
$c(z)=p$. For the last statement note that $w,z\in X(u)$.
\end{proof}

\paragraph{The gadget $A$.} We have $A(0)=\{4,6,7,8\}$, $A(1)\setminus A(0)=\{5\}$
and $\{5,6,7\}\subseteq A(1)\subseteq\{5,6,7,8\}$. So $A=A(0)\cup A(1)$. The
arguments below hold for both possible forms of $N[8]$.

\begin{lemma}\label{LA}
\textup{(a)} The set $A\cup\{0,1\}$ receives all three colours.

\textup{(b)} If a colour $d\notin\{c_0,c_1\}$ is missing on $A(0)$, then
$c(A(1))$ contains both colours different from $c_0$.

\textup{(c)} If $c_0=c_1$ and a colour $d\ne c_0$ is missing on $A(1)$, then
$c(A(0))$ contains both colours different from $c_0$.
\end{lemma}

\begin{proof}
(a) Suppose that only two colours appear on $A\cup\{0,1\}$. First let
$c_0=c_1=a$, and let $b$ be the other colour. Then $N[6]=\{0,1,6,8\}$ and
$N[7]=\{0,1,4,7\}$ force $c_6\neq c_8$ and $c_4\neq c_7$. So $a$ appears at least
twice in $N[8]$, on $0$ and on $\{6,8\}$. Hence $b$ is unique in $N[8]$. This
forces $c_4=c_5=a$, and so $c_7=b$. Now $N[4]=\{0,4,7,8\}$ is coloured
$a,a,b,c_8$. This forces $c_8=a$, and then $N[5]=\{1,5,8\}$ is monochromatic. Now
let $c_0=a\neq c_1=b$. Then $N[6]$ and $N[7]$ force $c_6=c_8$ and $c_4=c_7$. The
set $N[4]$ is coloured $a,c_4,c_4,c_8$, so $c_8=b$. Then $c_6=b$, and $N[5]$
forces $c_5=a$. So $N[8]$ contains $0,5$ of colour $a$ and $6,8$ of colour $b$,
and it is not good.

(b) Let $a=c_0$, and let $b$ be the colour different from $a$ and $d$. By (a),
$d$ appears on $A$. Hence $d$ appears on $5$, the only vertex of
$A\setminus A(0)$. Suppose that $b\notin\{c_6,c_7\}$. Since $6,7\in A(0)$, we get
$c_6=c_7=a$, and $c_1,c_4,c_8\in\{a,b\}$. Now $a$ appears at least twice in each
of $N[6]=\{0,1,6,8\}$, $N[7]=\{0,1,4,7\}$ and $N[4]=\{0,4,7,8\}$. So each of these
sets is good only if exactly one of its two remaining vertices has colour $b$.
These vertices are $1,8$ for $N[6]$, $1,4$ for $N[7]$, and $4,8$ for $N[4]$. The
first two conditions give that $c_8=b$ if and only if $c_4=b$. This contradicts
the third condition. Hence $b\in\{c_6,c_7\}\subseteq c(A(1))$ and
$d=c_5\in c(A(1))$.

(c) Let $a=c_0=c_1$, and let $b$ be the colour different from $a$ and $d$. By
(a), $d$ appears on $A$. Since $5,6,7\in A(1)$, we have $d\in\{c_4,c_8\}$ and
$c_5,c_6,c_7\in\{a,b\}$. If $b\in\{c_6,c_7\}$, we are done, since
$4,6,7,8\in A(0)$. Otherwise $c_6=c_7=a$. Then $N[6]$ and $N[7]$ give $c_8\neq a$
and $c_4\ne a$. The set $N[4]$ is coloured $a,c_4,a,c_8$, so $c_4\neq c_8$. Thus
$\{c_4,c_8\}=\{b,d\}$.
\end{proof}

\begin{lemma}\label{LE}
If $c_1\neq c_3$ and $c_2\in\{c_1,c_3\}$, then the colour not in
$\{c_1,c_2,c_3\}$ appears on $\{23,24\}$.
\end{lemma}

\begin{proof}
We have $N[23]=\{1,2,23,24\}$ and $N[24]=\{2,3,23,24\}$. By the symmetry
$1\leftrightarrow3$, $23\leftrightarrow24$ we may assume that $c_1=c_2=p$ and
$c_3=q\neq p$. If $c_{23}=c_{24}$, then $N[23]$ is coloured $p,p,c_{23},c_{23}$
and is not good. If $\{c_{23},c_{24}\}=\{p,q\}$, then $N[24]$ is coloured
$p,q,p,q$ and is not good. So $c_{23}\ne c_{24}$ and
$\{c_{23},c_{24}\}\ne\{p,q\}$. Hence the third colour appears on $\{23,24\}$.
\end{proof}

\begin{proof}[Proof of Theorem~\ref{T25}]
We distinguish cases according to the colours of the hubs.

\smallskip\noindent\textbf{Case 1: $c_1=c_3$.} Since $C$ is a copy of $W(1,3)$
attached to the edge $13$, Lemmas~\ref{L_rim} and~\ref{L_force} give a colour
$d\ne c_1$ that does not appear on $N(1)\setminus C$. So
$d\notin c(A(1))\cup c(D(1))$. Also $d\ne c_0$, since $0\in N(1)\setminus C$.

\emph{Case 1.1: $c_0\neq c_1$.} Let $a=c_0$ and $b=c_1=c_3$. So $d$ is the third
colour. By Lemma~\ref{LA}(a), $d\in c(A(0))$. By Lemma~\ref{LD}(b) for $D$,
$c_{21}=d$ and $c_{22}=a$. So $\mu_0(d)\ge2$ (on $A(0)$ and $21$), $\mu_0(a)\ge2$
(vertices $0,22$) and $\mu_0(b)\ge2$ (vertices $1,3$). Hence $N[0]$ is not good.

\emph{Case 1.2: $c_0=c_1=c_3=a$.} By Lemmas~\ref{LA}(c) and~\ref{LD}(b), both
colours different from $a$ appear on $A(0)$ and on $D(0)$. So they appear at
least twice in $N[0]$. Since $\mu_0(a)\ge3$, $N[0]$ is not good.

\smallskip\noindent\textbf{Case 2: $c_1\ne c_3$ and $c_2\in\{c_1,c_3\}$.} Let $q$
be the colour in $\{c_1,c_3\}\setminus\{c_2\}$, and let $d$ be the colour not in
$\{c_1,c_2,c_3\}$. By Lemma~\ref{LE}, $d\in c(\{23,24\})\subseteq c(N(2))$.

\emph{Case 2.1: $c_0=d$.} Then $\mu_2(d)\ge2$ (vertex $0$ and $E$) and
$\mu_2(c_2)\ge2$ (vertex $2$ and a vertex of $\{1,3\}$). So $\mu_2(q)=1$ and
$q\notin c(B(2))$. By Lemma~\ref{LD}(b) for $B$, $c_{12}=q$ and $c_{13}=d$. Then
$\mu_0(d)\ge2$ (vertices $0,13$), $\mu_0(q)\ge2$ and $\mu_0(c_2)\ge2$. So $N[0]$
is not good.

\emph{Case 2.2: $c_0=c_2=a$.} Let $b=q$. The hubs are coloured $a,a,a,b$, so
$\mu_i(a)\ge3$ for every hub $i$. Since $N[2]$ is good, $\mu_2(b)=1$ or
$\mu_2(d)=1$. In the first case $b\notin c(B(2))$. In the second case
$d\notin c(B(2))$, since $d$ appears on $E\subseteq N(2)$. So Lemma~\ref{LD}(b)
for $B$ shows that $b$ and $d$ both appear on $B(0)$. Thus $\mu_0(b)\ge2$, so
$\mu_0(d)=1$ and $d\notin c(A(0))\cup c(D(0))$. Since $c_1\in\{a,b\}$, we have
$d\notin\{c_0,c_1\}$. So Lemma~\ref{LA}(b) shows that $b$ and $d$ appear on
$A(1)$. Lemma~\ref{LD}(a) shows that $d$ appears on $D(1)$. Hence
$\mu_1(b)\ge2$, $\mu_1(d)\ge2$, $\mu_1(a)\ge3$, and $N[1]$ is not good.

\emph{Case 2.3: $c_0=q$.} Let $a=q$ and $b=c_2$. The hubs are coloured
$a,a,b,b$. So $\mu_i(d)=1$ for every hub $i$, and $d\notin\{c_0,c_1\}$. If $d$
appears on both $A(1)$ and $D(1)$, then $\mu_1(d)\ge2$. If $d$ is missing on both,
then by Lemmas~\ref{LA}(a) and~\ref{LD}(a) it appears on $A(0)$ and on $D(0)$, and
$\mu_0(d)\ge2$. So there is $X\in\{A,D\}$ with $d\notin c(X(1))$, while $d$
appears on $Y(1)$ for the other gadget $Y\in\{A,D\}$. By Lemma~\ref{LA}(a) or
Lemma~\ref{LD}(a), $d$ appears on $X(0)$. This is the only vertex of colour $d$
in $N[0]$, so $d\notin c(B(0))$. The only vertex of colour $d$ in $N[1]$ lies in
$Y(1)$. So $c_{23}\ne d$, and Lemma~\ref{LE} gives $c_{24}=d$. This is the only
vertex of colour $d$ in $N[2]$, so $d\notin c(B(2))$. Hence $d$ is missing on
$B\cup\{0,2\}$, contrary to Lemma~\ref{LD}(a) for $B$.

\smallskip\noindent\textbf{Case 3: $c_1,c_2,c_3$ are pairwise distinct.} Then
$\{c_1,c_2,c_3\}=\{1,2,3\}$, so $c_0\in\{c_1,c_2,c_3\}$. In each subcase the hubs
are coloured $a,a,b,d$ with $\{a,b,d\}=\{1,2,3\}$. So $\mu_i(a)\ge2$ for every
hub $i$. Hence $N[i]$ is good only if $b$ or $d$ is missing on
$N(i)\setminus\{0,1,2,3\}$.

\emph{Case 3.1: $c_0=c_1=a$.} Suppose that $b$ or $d$ is missing on $A(0)$. Then
Lemma~\ref{LA}(b), applied to this colour, shows that both $b$ and $d$ appear on
$A(1)$. So $N[1]$ is not good. Otherwise both $b$ and $d$ appear on $A(0)$, and
$N[0]$ is not good.

\emph{Case 3.2: $c_0=c_2=a$.} By Lemma~\ref{LD}(b) for $B$, both $b$ and $d$
appear on $B(0)$ or both appear on $B(2)$. So $N[0]$ or $N[2]$ is not good.

\emph{Case 3.3: $c_0=c_3=a$, $c_1=b$, $c_2=d$.} If $d\notin c(A(0))$, then by
Lemma~\ref{LA}(b) both $b$ and $d$ appear on $A(1)$, and $N[1]$ is not good. So
$d$ appears on $A(0)$, and therefore $b\notin c(B(0))$. If $b\notin c(B(2))$,
then Lemma~\ref{LD}(b) for $B$ gives $c_{12}=b$. This is a contradiction, since
$12\in B(0)$. Hence $b\in c(B(2))$, so $\mu_2(b)\ge2$. Since $\mu_2(a)\ge2$ and
$N[2]$ is good, we get $\mu_2(d)=1$, that is, $d\notin c(B(2))$. Since
$9,13\in B(0)\cap B(2)$, we get $c_9=c_{13}=a$. Then $N[12]$ gives $c_{12}\ne a$,
so $c_{12}=d$. Now $N[9]$ is coloured $a,d,a,d,a$ and is not good.
\end{proof}

The proof of Theorem~\ref{T25} never uses that $N[3]$ is good. This gives the
following.

\begin{proposition}\label{P_induced}
Every planar graph $H$ is an induced subgraph of a planar graph $G$ with
$\cf(G)=4$. If $H$ is connected, then $G$ can be chosen connected.
\end{proposition}

\begin{proof}
Identify a vertex of $H$ with the hub $3$ of $G_{25}$. The resulting graph $G$ is
planar and contains $H$ as an induced subgraph. It is connected if $H$ is
connected. Moreover, $N_G[v]=N_{G_{25}}[v]$ for every vertex $v\ne3$ of $G_{25}$.
So the proof of Theorem~\ref{T25} shows that $\cf(G)=4$.
\end{proof}

\subsection{Critical graphs on 25 vertices}\label{SecCrit25}

By Theorem~\ref{T25}, $G_{25}$ is not critical, as deleting any of the edges
$0\,18$, $0\,20$, $1\,8$ keeps $\cf=4$.

\begin{proposition}\label{P_G51}
$\cf(G''_{25})=4$.
\end{proposition}

\begin{proof}
Since $G''_{25}$ is planar, $\cf(G''_{25})\le4$. The other gadgets of $G''_{25}$
are as in $G'_{25}$. The proof of Theorem~\ref{T25} uses the gadget $A$ only
through Lemma~\ref{LA} and the equality $A=A(0)\cup A(1)$. In the new gadget,
$A(0)=\{4,7,8\}$ and $A(1)=\{4,5,6\}$. So it suffices to prove Lemma~\ref{LA} for
this gadget. Let $c$ be a $3$-CF-colouring of $G''_{25}$.

(a) Suppose that only two colours appear on $A\cup\{0,1\}$. First let
$c_0=c_1=a$, and let $b$ be the other colour. Then $N[4]=\{0,1,4,7,8\}$ forces
exactly one vertex of colour $b$ on $\{4,7,8\}$. Since $N[7]=\{0,4,7\}$ is not
monochromatic, this vertex is $4$ or $7$. So $c_8=a$. Then $N[5]=\{1,5,8\}$ and
$N[6]=\{1,6,8\}$ force $c_5=c_6=b$. Now $N[8]=\{0,4,5,6,8\}$ is coloured
$a,c_4,b,b,a$ and is not good. Now let $c_0=a\ne c_1=b$. Then $N[4]$ forces
$c_4=c_7=c_8$, and $N[7]$ forces this colour to be $b$. Further, $N[5]$ and
$N[6]$ force $c_5=c_6=a$. So $N[8]$ is coloured $a,b,a,a,b$ and is not good.

(b) Let $a=c_0$, and let $b$ be the colour different from $a$ and $d$. Suppose
that $b\notin c(A(1))$. By (a), $d\in\{c_5,c_6\}$. Moreover, $c_4=a$. Now $N[7]$
gives $c_7=b$. The set $N[4]$ is coloured $a,c_1,a,b,c_8$ with
$c_1,c_8\in\{a,b\}$. This gives $c_1=c_8=a$. So $N[8]$ is coloured
$a,a,c_5,c_6,a$ with $c_5,c_6\in\{a,d\}$. This forces exactly one of $c_5,c_6$ to
be $a$. Then $N[5]$ or $N[6]$ is monochromatic. Hence $b\in c(A(1))$, and
$d\in c(A(1))$ by (a).

(c) Let $a=c_0=c_1$, and let $b$ be the colour different from $a$ and $d$.
Suppose that $b\notin c(A(0))$. By (a), $d\in\{c_7,c_8\}$. Moreover, $c_4=a$ and
$c_7,c_8\in\{a,d\}$. Now $N[7]$ gives $c_7=d$. The set $N[4]$ is coloured
$a,a,a,d,c_8$, so $c_8=a$. Then $N[5]$ and $N[6]$ give $c_5=c_6=b$. So $N[8]$ is
coloured $a,a,b,b,a$ and is not good. Hence $b\in c(A(0))$, and $d\in c(A(0))$
by (a).
\end{proof}

\begin{theorem}\label{T25crit}
The graphs $G'_{25}$ and $G''_{25}$ are critical.
\end{theorem}

\begin{proof}
By Theorem~\ref{T25} and Proposition~\ref{P_G51},
$\cf(G'_{25})=\cf(G''_{25})=4$. In both graphs, the vertices $10$ and $11$ have
the same neighbourhoods. The same holds for $14$ and $15$, for $19$ and $20$, and
in $G''_{25}$ also for $5$ and $6$. So exchanging two such vertices is an
automorphism. Let $G\in\{G'_{25},G''_{25}\}$. Up to these automorphisms,
Tables~\ref{TabCert1} and~\ref{TabCert2} in Appendix~\ref{AppA} give a colouring
of $G-f$ with colours $1,2,3$ for every vertex and every edge $f$ of $G$. One
checks directly that it is a CF-colouring of $G-f$.
\end{proof}

\section{Infinite families}\label{SecFam}

Theorem~\ref{Tm} gives infinitely many counterexamples, since infinitely many
planar graphs have chromatic number $4$, for instance the odd wheels. By
Proposition~\ref{P_induced}, such graphs can contain any planar graph as an
induced subgraph. In this section we give two further families, triangulations
and critical graphs.

\subsection{Triangulations}\label{SecTri}

Let $T$ be the plane triangulation on the nine vertices $x,y,z,u_1,\dots,u_6$
with outer face $xyz$ (see Figure~\ref{FigT}). The vertex $x$ is adjacent to all
other vertices, and the neighbours of $x$ form the cycle
$y\,u_1u_2u_3u_4u_5u_6\,z\,y$. The remaining edges are $u_6y$, $u_6u_1$,
$u_6u_2$, $u_2u_4$ and $u_4u_6$. We regard $T$ as a gadget with apex $x$ and
inner vertices $u_1,\dots,u_6$. Thus
\[
N[u_1]=\{x,y,u_1,u_2,u_6\},\ N[u_2]=\{x,u_1,u_2,u_3,u_4,u_6\},\
N[u_3]=\{x,u_2,u_3,u_4\},
\]
\[
N[u_4]=\{x,u_2,u_3,u_4,u_5,u_6\},\ N[u_5]=\{x,u_4,u_5,u_6\}.
\]
The set $N[u_6]$ will not be needed.

\begin{figure}[t]
\centering
\begin{tikzpicture}[scale=0.5,yscale=-1,v/.style={circle,draw,fill=white,inner sep=1.4pt}]
\coordinate (X) at (0,-3); \coordinate (Y) at (9,6.5); \coordinate (Z) at (-9,6.5);
\foreach \i/\p in {1/3.75,2/2.25,3/0.75,4/-0.75,5/-2.25,6/-3.75}
  \coordinate (U\i) at (\p,2.6);
\draw (X)--(Y)--(Z)--(X);
\foreach \i in {1,...,6} \draw (X)--(U\i);
\draw (Y)--(U1)--(U2)--(U3)--(U4)--(U5)--(U6)--(Z);
\draw (U4) .. controls (-0.75,3.3) and (1.5,3.4) .. (U2);
\draw (U6) .. controls (-2.9,3.45) and (-1.4,3.4) .. (U4);
\draw (U6) .. controls (-2.9,4.4) and (2.25,4.4) .. (U2);
\draw (U6) .. controls (-3.17,5.34) and (3.6,5.4) .. (U1);
\draw (U6) .. controls (-3.62,6.2) and (3.6,6.4) .. (Y);
\node[v,label=above:$x$] at (X) {};
\node[v,label=right:$y$] at (Y) {};
\node[v,label=left:$z$] at (Z) {};
\foreach \i in {1,...,6} \node[v] at (U\i) {};
\foreach \i/\p in {2/2.25,3/0.75}
  \node[anchor=south west,inner sep=1pt] at (\p+0.12,2.45) {\scriptsize$u_{\i}$};
\foreach \i/\p in {4/-0.75,5/-2.25}
  \node[anchor=south east,inner sep=1pt] at (\p-0.12,2.45) {\scriptsize$u_{\i}$};
\node[anchor=east,inner sep=1pt] at (-3.95,2.6) {\scriptsize$u_6$};
\node[anchor=west,inner sep=1pt] at (3.95,2.6) {\scriptsize$u_1$};
\end{tikzpicture}
\caption{The gadget $T$, a plane triangulation with outer face $xyz$.}
\label{FigT}
\end{figure}
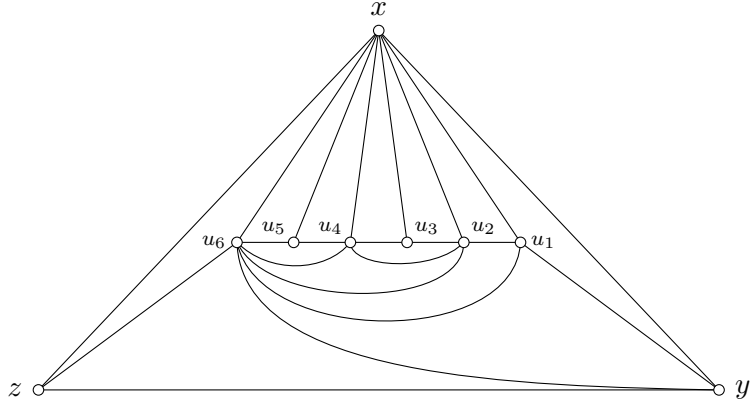

\begin{lemma}\label{LT}
$T$ is a forcing gadget with apex $x$.
\end{lemma}

\begin{proof}
Let $c$ be a colouring of $T$ with colours $1,2,3$ such that $c(x)=c(y)=a$ and
$N[u_1],\dots,N[u_5]$ are good. We show that both colours different from $a$
appear on $U=\{u_1,\dots,u_6\}\subseteq N(x)$. Suppose that some colour $d\ne a$
is missing on $U$. Then $U$ is coloured only by $a$ and the remaining colour $b$.
For $S\subseteq U$ let $\alpha(S)$ be the number of vertices of $S$ of colour $a$.
Since $a$ appears on $x$ and $y$, the colour $b$ is unique in $N[u_1]$. So
$\alpha(\{u_1,u_2,u_6\})=2$. The set $N[u_2]$ has six vertices coloured by $a$
and $b$. So it is good if and only if
$1+\alpha(\{u_1,u_2,u_3,u_4,u_6\})\in\{1,5\}$. This number is at least $3$, so it
equals $5$, and $c(u_3)=c(u_4)=a$. Similarly, $N[u_3]$ is good only if
$\alpha(\{u_2,u_3,u_4\})\in\{0,2\}$. So $c(u_2)=b$, and then
$c(u_1)=c(u_6)=a$. Also $N[u_5]$ is good only if
$\alpha(\{u_4,u_5,u_6\})\in\{0,2\}$, so $c(u_5)=b$. Now $N[u_4]$ is coloured
$a,b,a,a,b,a$, which is not good. This is a contradiction.
\end{proof}

Let $H$ be a $2$-connected plane graph. Insert a new vertex $w_F$ into every face
$F$ of $H$, and join it to all vertices of $F$. Since $H$ is $2$-connected, every
face of $H$ is bounded by a cycle. Also every edge $xy$ of $H$ lies on two
distinct faces $F$ and $F'$. So the result is a simple plane triangulation. In
it, $xy$ lies in exactly two faces $w_Fxy$ and $w_{F'}xy$, and these $2|E(H)|$
faces are pairwise distinct. For every edge of $H$, choose one end $x$. Insert a
copy of $T$ into each of the two faces $w_Fxy$ and $w_{F'}xy$. Its outer face
$xyz$ is identified with $x,y,w_F$ and with $x,y,w_{F'}$, respectively. Let $G(H)$
be the resulting plane triangulation. It has
$|V(H)|+|F(H)|+12|E(H)|=13|E(H)|+2$ vertices.

\begin{theorem}\label{Tm_tri}
For every $2$-connected plane graph $H$ with $\chi(H)=4$, the plane triangulation
$G(H)$ satisfies $\cf(G(H))=4$.
\end{theorem}

\begin{proof}
We have $\cf(G(H))\le4$ by the Four Colour Theorem. The two copies of $T$
inserted at an edge $xy$ of $H$ are attached to $xy$, and they have disjoint sets
of inner vertices. So by Lemmas~\ref{LT} and~\ref{L_force} every
$3$-CF-colouring of $G(H)$ is a proper $3$-colouring of $H$. But such a colouring
does not exist.
\end{proof}

Theorem~\ref{Tm_tri} applies, for instance, to every odd wheel and to every plane
triangulation with a vertex of odd degree. Indeed, a plane triangulation is
$3$-colourable if and only if all its vertices have even degree (Heawood, see
\cite{TsaiWest}). For $H=K_4$ the triangulation $G(H)$ has $80$ vertices.

\subsection{Critical graphs}\label{SecCritFam}

By Proposition~\ref{P_noncrit}, the graphs of Theorem~\ref{Tm} are not critical.
Its proof shows the following. A single forcing gadget on an edge $xy$ with
$c(x)=c(y)$ already gives a contradiction if both colours different from $c(x)$
appear on $N(x)$ outside the gadget. The graphs $H_k$ below have two gadgets on
each edge of an odd cycle. Each spoke has only one gadget, and there this
situation occurs.

For $k\ge1$ let $n=2k+1$, and let $\W_n$ be the odd wheel with centre $h$ and
cycle $C=v_1v_2\cdots v_n$ (indices modulo $n$). Let $H_k$ be obtained from $\W_n$
by attaching two copies of $W(v_i,v_{i+1})$ to every cycle edge $v_iv_{i+1}$ and
one copy of $W(v_i,h)$ to every spoke $hv_i$. Thus all apices lie on $C$. The two
copies of $W(v_i,v_{i+1})$ are $R_i^1$ and $R_i^2$. The rim of $R_i^t$ is
$r^{i,t}_1r^{i,t}_2r^{i,t}_3r^{i,t}_4$, where $v_i$ is adjacent to all four rim
vertices and $v_{i+1}$ only to $r^{i,t}_4$. The copy of $W(v_i,h)$ is $S_i$. Its
rim is $s^i_1s^i_2s^i_3s^i_4$, where $v_i$ is adjacent to all four rim vertices
and $h$ only to $s^i_4$. Each gadget can be drawn close to its edge, so $H_k$ is
planar. It has $26k+14$ vertices and $58k+29$ edges. In $H_k$ we have
\begin{equation}\label{eq_loc}
N[v_i]=\{h,v_{i-1},v_i,v_{i+1},r^{i-1,1}_4,r^{i-1,2}_4\}\cup
\{s^i_j,r^{i,1}_j,r^{i,2}_j:1\le j\le4\}.
\end{equation}

\begin{theorem}\label{Tm_Hk}
For every $k\ge1$, $\cf(H_k)=4$.
\end{theorem}

\begin{proof}
Let $c$ be a $3$-CF-colouring of $H_k$. By Lemmas~\ref{L_rim}
and~\ref{L_force}, the two gadgets on every cycle edge force
$c(v_i)\ne c(v_{i+1})$. So the odd cycle $C$ receives all three colours. In
particular, $a=c(h)$ appears at least twice in $N[h]$. Since $N[h]$ is good, some
colour $d\ne a$ is unique in $N[h]$. Since $d$ appears on $C$, it appears there
exactly once, say on $v_m$. Let $b$ be the third colour. The path $C-v_m$ has
$2k$ vertices and is properly coloured by $a$ and $b$. So its ends, the two
neighbours of $v_m$ on $C$, have different colours. Let $v_i$ be the one of
colour $a$. The gadget $S_i$ is attached to the spoke $hv_i$. By
Lemma~\ref{L_force} applied to $S_i$, one of $b,d$ does not appear on $N(v_i)$
outside the rim of $S_i$. But $v_m$ has colour $d$, the other neighbour of $v_i$
on $C$ has colour $b$, and both lie there. This is a contradiction.
\end{proof}

\begin{theorem}\label{Tm_crit}
For every $k\ge1$, the graph $H_k$ is critical.
\end{theorem}

We prove Theorem~\ref{Tm_crit} at the end of this section, since the proof needs a
reduction by automorphisms and three auxiliary lemmas. A vertex or an edge $f$ of
$H_k$ is a \emph{wheel element} if it is a vertex or an edge of $\W_n$. Otherwise
$f$ is a \emph{gadget element}. We first reduce to $22$ types of elements. Then we
treat wheel elements in Lemma~\ref{L_wheel} and gadget elements in
Lemma~\ref{L_base}.

\paragraph{Types of elements.} The following maps are automorphisms of $H_k$.
The first is the rotation $v_i\mapsto v_{i+1}$, $s^i_j\mapsto s^{i+1}_j$,
$r^{i,t}_j\mapsto r^{i+1,t}_j$, which fixes $h$. The second, for every $i$,
exchanges $r^{i,1}_j$ and $r^{i,2}_j$ for $j=1,\dots,4$. The third, for every
gadget, exchanges its rim vertices $r_1$ and $r_3$ and fixes all other vertices.
This is an automorphism since $r_1$ and $r_3$ have the same neighbours
$r_2,r_4,x$. Using these automorphisms, every wheel element is equivalent to one
of $h$, $v_1$, $hv_1$, $v_1v_2$. Every gadget element is equivalent to an element
of $S_1$ or of $R_1^1$ of one of the nine forms
\[
r_1,\ r_2,\ r_4,\ r_1r_2,\ r_4r_1,\ xr_1,\ xr_2,\ xr_4,\ yr_4.
\]
Here $r_1r_2r_3r_4$ is the rim of the gadget, $x=v_1$ is its apex and $y$ is its
other attachment vertex. So $y=h$ for $S_1$ and $y=v_2$ for $R_1^1$. The elements
$r_3$, $r_2r_3$, $r_3r_4$ and $xr_3$ are equivalent to $r_1$, $r_1r_2$, $r_4r_1$
and $xr_1$, respectively. This gives $4+18=22$ types.

\begin{lemma}\label{L_wheel}
For every $k\ge1$ and every wheel element $f$ of $H_k$, $\cf(H_k-f)\le3$.
\end{lemma}

\begin{proof}
In each case we colour the vertices of $\W_n-f$ and colour a few gadgets
explicitly. Every other gadget $W(x,y)$ is coloured standardly. This is possible
since $x$ and $y$ will be present and $c(x)\ne c(y)$. In each case we check that
no vertex of $\W_n-f$ has a neighbour of its own colour. Then its closed
neighbourhood is good. For $2\le j\le n$ let $\pi(j)=1$ if $j$ is even and
$\pi(j)=2$ if $j$ is odd. Note that $\pi(n)=2$.

\emph{(a) $f=h$.} Let $c(v_1)=3$ and $c(v_j)=\pi(j)$ for $2\le j\le n$. This is a
proper colouring of $C$. The gadgets $R_i^t$ are standard. Each $S_i$ has lost
$h$. Its rim gets one colour different from $c(v_i)$, so its rim vertices are
fine by (W1). No $v_i$ has a neighbour of its colour. Indeed, $C$ is properly
coloured, and all rim vertices adjacent to $v_i$ avoid $c(v_i)$. These are the
rim vertices of $S_i$ and $R_i^t$ and the vertices $r^{i-1,t}_4$.

\emph{(b) $f=v_1$.} Let $c(h)=3$ and $c(v_j)=\pi(j)$ for $2\le j\le n$. The
gadgets $S_1,R_1^1,R_1^2$ have lost their apex. Let $y$ be the remaining
attachment vertex, so $y$ is $h$, $v_2$, $v_2$, respectively. Let $p,q$ be the two
colours different from $c(y)$, in either order. Colour $r_1,r_2$ by $p$ and
$r_3,r_4$ by $q$. Then $N[r_1]=\{r_4,r_1,r_2\}$, $N[r_2]=\{r_1,r_2,r_3\}$ and
$N[r_3]=\{r_2,r_3,r_4\}$ are not monochromatic, and $p$ is unique in
$N[r_4]=\{r_3,r_4,r_1,y\}$. The gadgets $R_n^1,R_n^2$ have lost $v_1$. Their rims
get colour $3\ne c(v_n)$, which is fine by (W1). All other gadgets are standard.
No neighbour of $h$ has colour $3$. Indeed, the $v_j$ have colours $1,2$, the
vertices $s^j_4$ with $j\ge2$ avoid $3$, and $c(s^1_4)=q\ne3$. No neighbour of
$v_2$ has colour $c(v_2)$, since $c(r^{1,t}_4)=q\ne c(v_2)$. No neighbour of
$v_n$ has colour $c(v_n)$, since the rims of $R_n^t$ have colour $3$. The other
$v_j$ are as in (a).

\emph{(c) $f=hv_1$.} Let $c(h)=c(v_1)=3$ and $c(v_j)=\pi(j)$ for $2\le j\le n$.
Colour the rim $s^1_1s^1_2s^1_3s^1_4$ of $S_1$ by $1,2,2,2$, which is fine by
(W2). All other gadgets are standard, since their attachment vertices have
distinct colours. The edge $hv_1$ is deleted, and the rim of $S_1$ avoids $3$. So
neither $h$ nor $v_1$ has a neighbour of colour $3$. The other $v_j$ are as in
(a).

\emph{(d) $f=v_1v_2$.} Let $c(h)=3$, $c(v_1)=1$ and $c(v_j)=\pi(j)$ for
$2\le j\le n$. So $c(v_1)=c(v_2)=1$ and $c(v_n)=2$. Colour the rims of $R_1^1$ and
$R_1^2$ by $2,3,3,3$, which is fine by (W2). All other gadgets are standard. The
edge $v_1v_2$ is deleted, and the rims of $R_1^t$ avoid $1$. So neither $v_1$ nor
$v_2$ has a neighbour of colour $1$. The vertex $h$ is as in (b), and the other
$v_j$ are as in (a).
\end{proof}

\paragraph{Gadget elements.} Let $e$ be a gadget element of $S_1$ or of $R_1^1$.
Consider the gadget that contains $e$, and the edge that carries this gadget. By
the argument in the proof of Theorem~\ref{Tm_Hk}, this edge is monochromatic in
every $3$-CF-colouring of $H_k-e$. So we first colour such a damaged gadget.

\begin{lemma}\label{L_dam}
Let $f$ be one of the nine elements $r_1$, $r_2$, $r_4$, $r_1r_2$, $r_4r_1$,
$xr_1$, $xr_2$, $xr_4$, $yr_4$ of $W(x,y)$. Let $\{a,b,d\}=\{1,2,3\}$, and let
$c(x)=c(y)=a$. Colour the rim $r_1r_2r_3r_4$ of $W(x,y)-f$ by the word $w_f$ given
by
\[
\begin{array}{c|ccccccccc}
f & r_1 & r_2 & r_4 & r_1r_2 & r_4r_1 & xr_1 & xr_2 & xr_4 & yr_4\\ \hline
w_f & {-}aba & a{-}ab & bbb{-} & bbaa & baab & dbbb & adab & bbbb & bbbb
\end{array}
\]
A dash marks a deleted vertex. Then in $W(x,y)-f$ the closed neighbourhood of
every rim vertex is good. Moreover, every rim vertex adjacent to $x$ or to $y$
has colour $a$ or $b$.
\end{lemma}

\begin{proof}
The colour $d$ is used only on $r_1$ for $f=xr_1$ and on $r_2$ for $f=xr_2$. In
both cases the edge from $x$ to this vertex is deleted, and $y$ is not adjacent
to it. This proves the second statement. For $f\in\{r_4,xr_4,yr_4\}$, every
closed neighbourhood of a rim vertex contains exactly one vertex of colour $a$,
and all its other vertices have colour $b$. In the remaining cases, the colours
on the closed neighbourhoods of the rim vertices are as follows.
\begin{align*}
f=r_1:&\ N[r_2]=\{x,r_2,r_3\}{:}\ aab,\ N[r_3]{:}\ aaba,\ N[r_4]=\{x,y,r_3,r_4\}{:}\ aaba;\\
f=r_2:&\ N[r_1]=\{x,r_4,r_1\}{:}\ aba,\ N[r_3]=\{x,r_3,r_4\}{:}\ aab,\ N[r_4]{:}\ aaaba;\\
f=r_1r_2:&\ N[r_1]=\{x,r_4,r_1\}{:}\ aab,\ N[r_2]=\{x,r_2,r_3\}{:}\ aba,\ N[r_3]{:}\ abaa,\ N[r_4]{:}\ aaaab;\\
f=r_4r_1:&\ N[r_1]=\{x,r_1,r_2\}{:}\ aba,\ N[r_2]{:}\ abaa,\ N[r_3]{:}\ aaab,\ N[r_4]=\{x,y,r_3,r_4\}{:}\ aaab;\\
f=xr_1:&\ N[r_1]=\{r_4,r_1,r_2\}{:}\ bdb,\ N[r_2]{:}\ adbb,\ N[r_3]{:}\ abbb,\ N[r_4]{:}\ aabbd;\\
f=xr_2:&\ N[r_1]{:}\ abad,\ N[r_2]=\{r_1,r_2,r_3\}{:}\ ada,\ N[r_3]{:}\ adab,\ N[r_4]{:}\ aaaba.
\end{align*}
In every case some colour is unique.
\end{proof}

\paragraph{The colourings.} Let $e$ be a gadget element of $S_1$ or of $R_1^1$ of
one of the nine forms $f$ of Lemma~\ref{L_dam}. Recall that $x=v_1$. Also $y=h$
if $e\in S_1$, and $y=v_2$ if $e\in R_1^1$. We define a colouring $c$ of $H_k-e$
with $c(h)=2$. On $C$ let
\begin{align*}
&c(v_i)=2\ (i\text{ odd},\ i\le2k-1),\quad c(v_i)=3\ (i\text{ even}),\quad
c(v_n)=1 &&\text{if }e\in S_1,\\
&c(v_1)=c(v_2)=3,\quad c(v_i)=2\ (i\ge3\text{ odd}),\quad
c(v_i)=3\ (i\ge4\text{ even}) &&\text{if }e\in R_1^1.
\end{align*}
Thus $C$ is properly coloured, except for the edge $v_1v_2$ if $e\in R_1^1$. The
spokes $hv_i$ with $c(v_i)=2$ are monochromatic. The gadget containing $e$ is
attached to a monochromatic edge. Its rim is coloured by the word $w_f$ of
Lemma~\ref{L_dam} with $(a,b,d)=(2,3,1)$ if $e\in S_1$ and $(a,b,d)=(3,2,1)$ if
$e\in R_1^1$. Every other gadget $W(x',y')$ is coloured by the following rules.
Rim colours are listed in the order $r_1,\dots,r_4$.
\begin{description}
\item[(G1)] If $c(x')\ne c(y')$ and $c(x')\ne2$, all rim vertices get colour $2$.
The only exception is the rim of $S_2$ if $e\in R_1^1$, which gets $2,2,2,1$.
\item[(G2)] If $c(x')=2\ne c(y')$, the rim gets $2,3,2,2$. Here $c(y')=3$.
\item[(G3)] If $c(x')=c(y')$, the vertex $r_3$ gets colour $1$. The other rim
vertices get the colour different from $1$ and $c(x')$.
\end{description}
Rule (G3) applies to the gadgets $S_i$, $i\ne1$, with $c(v_i)=2$. If
$e\in R_1^1$, it also applies to $R_1^2$.

\begin{lemma}\label{L_base}
For every gadget element $e$ of one of the $18$ types and every $k\ge1$, the
colouring $c$ is a $3$-CF-colouring of $H_k-e$.
\end{lemma}

\begin{proof}
Rim vertices of the gadget containing $e$ have the same closed neighbourhoods in
$H_k-e$ as in $W(x,y)-e$. So their closed neighbourhoods are good by
Lemma~\ref{L_dam}. The rims coloured by (G1) and (G3) are as in (W1) and (W2),
respectively. Under (G2), the colours on $N[r_1]$ are $2,2,2,3$, on $N[r_2]$ they
are $2,2,3,2$, on $N[r_3]$ they are $2,3,2,2$, and on $N[r_4]$ they are
$2,3,2,2,2$. The colour $3$ is unique in each of these sets. It remains to
consider $h$ and the vertices $v_i$, using \eqref{eq_loc}.

Let $e\in S_1$. Let $i$ be even. Then $c(v_i)=3$ is unique in $N[v_i]$. Indeed,
the vertices $h,v_{i-1},v_{i+1}$ have colours $1,2$. The rim vertices of $S_i$,
$R_i^1$, $R_i^2$ have colour $2$ by (G1), and $r^{i-1,t}_4$ has colour $2$ by
(G2). Now let $i$ be odd with $3\le i\le2k-1$. Then $1$ is unique in $N[v_i]$,
and it appears on $s^i_3$. Indeed, the vertices $h,v_{i\pm1}$ have colours $2,3$.
The rim of $S_i$ is $3,3,1,3$ by (G3), and the rim of $R_i^t$ is $2,3,2,2$ by
(G2). The vertex $r^{i-1,t}_4$ has colour $2$ by (G1). The colour $c(v_n)=1$ is
unique in $N[v_n]$. Indeed, $h,v_{n-1},v_1$ have colours $2,3,2$, and all rim
vertices in $N[v_n]$ have colour $2$ by (G1). In $N[v_1]$ the colour $1$ appears
only on $v_n$. Indeed, $v_1$ and $h$ have colour $2$, and $v_2$ has colour $3$.
The rims of $R_1^1,R_1^2$ have colours $2,3$ by (G2), and $r^{n,t}_4$ has colour
$2$ by (G1). The rim vertices of $S_1-e$ adjacent to $v_1$ have colours $2,3$ by
Lemma~\ref{L_dam}. Finally, $1$ is unique in $N[h]$, and it appears on $v_n$.
Indeed, the other $v_i$ have colours $2,3$, and $c(s^i_4)\in\{2,3\}$ for $i\ge2$
by (G1) and (G3). If $s^1_4$ is adjacent to $h$, then $c(s^1_4)\in\{2,3\}$ by
Lemma~\ref{L_dam}.

Let $e\in R_1^1$. For $i\ge3$ the arguments for even and for odd $i$ are as
above. For $i=3$ note that $r^{2,t}_4$ has colour $2$ by (G1). For $i=n$ note
that $c(v_1)=3$. In $N[v_2]$ the colour $1$ appears only on $s^2_4$. Indeed,
$v_1,v_2$ have colour $3$, and $h,v_3$ have colour $2$. The rim of $S_2$ is
$2,2,2,1$, and the rims of $R_2^t$ have colour $2$. The vertex $r^{1,2}_4$ has
colour $2$ by (G3). If $r^{1,1}_4$ is adjacent to $v_2$, then it has colour $2$
or $3$ by Lemma~\ref{L_dam}. In $N[v_1]$ the colour $1$ appears only on
$r^{1,2}_3$. Indeed, $v_1,v_2$ have colour $3$, and $h,v_n$ have colour $2$. The
rim of $S_1$ has colour $2$ by (G1), and the rim of $R_1^2$ is $2,2,1,2$. The
vertex $r^{n,t}_4$ has colour $2$ by (G2). The rim vertices of $R_1^1-e$ adjacent
to $v_1$ have colours $2,3$. Finally, $1$ is unique in $N[h]$, and it appears on
$s^2_4$. Indeed, the $v_i$ have colours $2,3$, and $c(s^i_4)\in\{2,3\}$ for
$i\ne2$ by (G1) and (G3).
\end{proof}

\begin{proof}[Proof of Theorem~\ref{Tm_crit}]
By Theorem~\ref{Tm_Hk}, $\cf(H_k)=4$. So it remains to show that
$\cf(H_k-f)\le3$ for every vertex or edge $f$ of $H_k$. Automorphisms of $H_k$
preserve this property. So we may assume that $f$ is one of the $22$ types
listed above. If $f$ is a wheel element, then $\cf(H_k-f)\le3$ by
Lemma~\ref{L_wheel}. Otherwise $f$ is a gadget element of one of the $18$ types,
and $\cf(H_k-f)\le3$ by Lemma~\ref{L_base}.
\end{proof}

\vspace{2mm}\bigskip\noindent\textbf{AI declaration.}~~Anthropic Claude.ai was used in the preparation
of this manuscript. It assisted with the arguments and their exposition, and with
the verification of the proofs. All results were verified by the authors, who take
full responsibility for the content.

\vspace{2mm}\bigskip\noindent\textbf{Acknowledgements.}
R.~\v{S}krekovski acknowledges partial support from the Slovenian Research and
Innovation Agency program P1--0383, the bilateral Slovenian--Croatian project
BI-HR/25-27-004 and the annual work program of Rudolfovo. J.~Sedlar
acknowledges support from the Croatian Ministry of Science, Education and Youth
through the bilateral Croatian--Slovenian project 2025--26, and from the
NextGenerationEU instrument via IP-UNIST-17 (GEORAZ).

\appendix
\section{Colourings for Theorem~\ref{T25crit}}\label{AppA}

Each entry of Tables~\ref{TabCert1} and~\ref{TabCert2} consists of a vertex or an edge $f$
of $G$ and a colouring of $G-f$. The colouring is given by the colours of the
vertices $0,1,\dots,24$, grouped as $0$--$3$ (hubs), $A$, $B$, $C$, $D$, $E$. A
dash marks the deleted vertex.

\begin{table}[t]
\centering\scriptsize\setlength{\tabcolsep}{2.5pt}
\begin{tabular}{@{}rl@{\qquad}rl@{\qquad}rl@{}}
\hline
$f$ & colouring of $G-f$ & $f$ & colouring of $G-f$ & $f$ & colouring of $G-f$\\
\hline
$0$ & \texttt{-123\,33331\,33311\,3233\,23313\,31} & $1$ & \texttt{2-33\,33331\,33311\,3231\,23311\,31} & $2$ & \texttt{21-3\,33331\,33311\,3233\,23313\,31}\\
$3$ & \texttt{211-\,33331\,33311\,3333\,13313\,31} & $4$ & \texttt{3123\,-2333\,33311\,3233\,23313\,31} & $5$ & \texttt{3233\,2-222\,33312\,3333\,23132\,31}\\
$6$ & \texttt{3123\,33-31\,33311\,3233\,23313\,31} & $7$ & \texttt{3123\,133-3\,33311\,3233\,23313\,31} & $8$ & \texttt{3233\,2332-\,33321\,3333\,13332\,31}\\
$9$ & \texttt{3233\,33331\,-3311\,3333\,23331\,31} & $10$ & \texttt{3233\,33331\,3-213\,3333\,33311\,31} & $12$ & \texttt{3233\,33331\,333-1\,3333\,33311\,31}\\
$13$ & \texttt{3213\,33331\,1331-\,3333\,33311\,33} & $14$ & \texttt{2333\,33331\,33311\,-233\,23311\,31} & $16$ & \texttt{2333\,33331\,33311\,33-2\,23311\,31}\\
$17$ & \texttt{2111\,33331\,33311\,333-\,13313\,31} & $18$ & \texttt{3313\,23331\,13211\,3213\,-3311\,33} & $19$ & \texttt{1111\,23331\,32311\,3233\,1-331\,31}\\
$21$ & \texttt{3233\,31232\,33321\,3333\,333-3\,31} & $22$ & \texttt{2333\,33331\,33311\,3231\,2221-\,31} & $23$ & \texttt{2133\,33331\,33311\,3233\,23313\,-1}\\
$24$ & \texttt{2133\,33331\,33311\,3233\,23313\,3-} & $0\,1$ & \texttt{3213\,33331\,33312\,3333\,33313\,32} & $0\,2$ & \texttt{3123\,33231\,33311\,3233\,23313\,31}\\
$0\,3$ & \texttt{3123\,23221\,13311\,3233\,33312\,31} & $0\,4$ & \texttt{3123\,23333\,33311\,3233\,23313\,31} & $0\,6$ & \texttt{3123\,33331\,33311\,3233\,23313\,31}\\
$0\,7$ & \texttt{3123\,13333\,33311\,3233\,23313\,31} & $0\,8$ & \texttt{3123\,33332\,33311\,3233\,23313\,31} & $0\,9$ & \texttt{3213\,33331\,23313\,3333\,13313\,32}\\
$0\,12$ & \texttt{3233\,33331\,33321\,3333\,23331\,31} & $0\,13$ & \texttt{3213\,33331\,33312\,3333\,33313\,32} & $0\,21$ & \texttt{3233\,31232\,33321\,3333\,33313\,31}\\
$0\,22$ & \texttt{2333\,33331\,33311\,3213\,33312\,31} & $1\,2$ & \texttt{2133\,33331\,33311\,3233\,23313\,32} & $1\,3$ & \texttt{2131\,33331\,33311\,3233\,23313\,33}\\
$1\,5$ & \texttt{3233\,31232\,33321\,3333\,23132\,31} & $1\,6$ & \texttt{3123\,33331\,33311\,3233\,23313\,31} & $1\,7$ & \texttt{3123\,13333\,33311\,3233\,23313\,31}\\
$1\,14$ & \texttt{2333\,33331\,33311\,1233\,23311\,31} & $1\,16$ & \texttt{2333\,33331\,33311\,3213\,23311\,31} & $1\,17$ & \texttt{2333\,33331\,33311\,3231\,23311\,31}\\
$1\,18$ & \texttt{1111\,23331\,33211\,3233\,23311\,31} & $1\,19$ & \texttt{2333\,33331\,33311\,3213\,21213\,31} & $1\,22$ & \texttt{2333\,33331\,33311\,3213\,23311\,31}\\
$1\,23$ & \texttt{2133\,33331\,33311\,3233\,23313\,31} & $2\,3$ & \texttt{3213\,33331\,12211\,3133\,33311\,11} & $2\,9$ & \texttt{3213\,33331\,33313\,3333\,33313\,32}\\
$2\,10$ & \texttt{3213\,33331\,12311\,3333\,33313\,33} & $2\,13$ & \texttt{2133\,33331\,33311\,3233\,23313\,32} & $2\,23$ & \texttt{2133\,33331\,33311\,3233\,23313\,31}\\
$2\,24$ & \texttt{2113\,33331\,33311\,3333\,13313\,31} & $3\,17$ & \texttt{2111\,33331\,33311\,3333\,33311\,31} & $3\,24$ & \texttt{2113\,33331\,33311\,3333\,33311\,31}\\
$4\,7$ & \texttt{3123\,13333\,33311\,3233\,23313\,31} & $4\,8$ & \texttt{3333\,13331\,31213\,3213\,33312\,31} & $5\,8$ & \texttt{3233\,23222\,33321\,3333\,23132\,31}\\
$6\,8$ & \texttt{3123\,33331\,33311\,3233\,23313\,31} & $9\,12$ & \texttt{3233\,33331\,33311\,3133\,33313\,31} & $9\,13$ & \texttt{3233\,33331\,33311\,3333\,13313\,31}\\
$10\,13$ & \texttt{3233\,33331\,31113\,3333\,23331\,31} & $14\,16$ & \texttt{2333\,33331\,33311\,2323\,23311\,31} & $14\,17$ & \texttt{2111\,33331\,33311\,3113\,33311\,31}\\
$18\,21$ & \texttt{3313\,23331\,13211\,3213\,33311\,33} & $18\,22$ & \texttt{3313\,23331\,13211\,3213\,13311\,33} & $19\,22$ & \texttt{1111\,23331\,33211\,3233\,13331\,31}\\
$23\,24$ & \texttt{2133\,33331\,33311\,3233\,23313\,31} &  &  &  & \\
\hline
\end{tabular}
\caption{Colourings certifying that $G'_{25}$ is critical.}
\label{TabCert1}
\medskip
\centering\scriptsize\setlength{\tabcolsep}{2.5pt}
\begin{tabular}{@{}rl@{\qquad}rl@{\qquad}rl@{}}
\hline
$f$ & colouring of $G-f$ & $f$ & colouring of $G-f$ & $f$ & colouring of $G-f$\\
\hline
$0$ & \texttt{-333\,33312\,32123\,3132\,22212\,31} & $1$ & \texttt{3-33\,33312\,32123\,3132\,21132\,31} & $2$ & \texttt{33-2\,33312\,32123\,3323\,22122\,31}\\
$3$ & \texttt{333-\,33312\,32123\,2222\,22122\,31} & $4$ & \texttt{3333\,-3322\,32123\,3213\,33312\,31} & $5$ & \texttt{3332\,3-223\,33312\,3323\,22122\,31}\\
$7$ & \texttt{3333\,333-2\,32123\,3132\,33312\,31} & $8$ & \texttt{3233\,2332-\,33312\,3333\,23132\,31} & $9$ & \texttt{3332\,33312\,-3322\,3323\,22122\,31}\\
$10$ & \texttt{3332\,33312\,3-123\,3323\,22122\,31} & $12$ & \texttt{3332\,33312\,333-2\,3323\,22122\,31} & $13$ & \texttt{3233\,33312\,3112-\,3333\,23132\,31}\\
$14$ & \texttt{3333\,33312\,32123\,-233\,22122\,31} & $16$ & \texttt{3333\,33312\,32123\,33-2\,22122\,31} & $17$ & \texttt{3333\,33312\,32123\,332-\,22122\,31}\\
$18$ & \texttt{3333\,33312\,32123\,2133\,-3322\,31} & $19$ & \texttt{3333\,33312\,32123\,3132\,3-123\,31} & $21$ & \texttt{3333\,33312\,32123\,3132\,333-2\,31}\\
$22$ & \texttt{3333\,33312\,32123\,3132\,2222-\,31} & $23$ & \texttt{3332\,33312\,32123\,3323\,22122\,-3} & $24$ & \texttt{3233\,33312\,32123\,3333\,23132\,3-}\\
$0\,1$ & \texttt{3233\,33311\,33321\,3333\,22113\,31} & $0\,2$ & \texttt{3223\,33312\,21122\,3333\,22123\,21} & $0\,3$ & \texttt{3123\,23311\,11121\,3333\,22222\,31}\\
$0\,4$ & \texttt{3233\,33322\,33312\,3333\,23132\,31} & $0\,7$ & \texttt{3333\,33312\,32123\,3132\,33312\,31} & $0\,8$ & \texttt{3333\,33312\,32113\,3132\,33312\,31}\\
$0\,9$ & \texttt{3123\,33322\,13123\,3333\,22332\,31} & $0\,12$ & \texttt{3233\,33312\,33312\,3333\,22123\,31} & $0\,13$ & \texttt{3332\,33312\,33321\,3323\,22122\,31}\\
$0\,21$ & \texttt{3333\,33312\,32123\,3132\,33312\,31} & $0\,22$ & \texttt{3333\,33312\,32123\,3132\,22222\,31} & $1\,2$ & \texttt{3233\,33312\,32123\,3333\,23132\,31}\\
$1\,3$ & \texttt{3131\,33312\,32113\,3332\,22213\,33} & $1\,4$ & \texttt{3133\,32313\,33312\,3333\,22313\,32} & $1\,5$ & \texttt{3233\,21322\,33312\,3333\,23132\,31}\\
$1\,14$ & \texttt{3333\,33312\,32123\,1233\,22122\,31} & $1\,16$ & \texttt{3333\,33312\,32123\,3312\,22122\,31} & $1\,17$ & \texttt{3333\,33312\,32123\,3231\,22122\,32}\\
$1\,18$ & \texttt{3333\,33312\,32123\,2313\,21133\,31} & $1\,19$ & \texttt{3333\,33312\,32123\,3132\,21322\,31} & $1\,22$ & \texttt{3333\,33312\,32123\,3132\,22222\,31}\\
$1\,23$ & \texttt{3233\,33312\,32123\,3333\,23132\,23} & $2\,3$ & \texttt{3332\,33312\,32123\,3323\,22122\,31} & $2\,9$ & \texttt{3233\,33312\,21133\,3333\,23132\,31}\\
$2\,10$ & \texttt{3233\,33312\,32123\,3333\,23132\,31} & $2\,13$ & \texttt{1332\,33332\,33122\,3323\,22322\,31} & $2\,23$ & \texttt{3133\,33312\,32113\,3332\,22313\,31}\\
$2\,24$ & \texttt{3332\,33312\,32123\,3323\,22122\,31} & $3\,17$ & \texttt{3333\,33312\,32123\,2222\,22122\,31} & $3\,24$ & \texttt{3332\,33312\,32123\,3323\,22122\,32}\\
$4\,7$ & \texttt{3333\,33322\,32123\,2133\,33312\,31} & $4\,8$ & \texttt{3332\,33322\,33312\,3323\,22122\,31} & $5\,8$ & \texttt{2233\,23332\,22212\,3333\,23132\,31}\\
$9\,12$ & \texttt{3332\,33312\,33322\,3323\,22122\,31} & $9\,13$ & \texttt{3332\,33312\,33322\,3323\,22122\,31} & $10\,13$ & \texttt{3332\,33312\,32223\,3323\,22122\,31}\\
$14\,16$ & \texttt{3333\,33312\,32123\,2323\,22122\,31} & $14\,17$ & \texttt{3333\,33312\,32123\,2332\,22122\,31} & $18\,21$ & \texttt{3333\,33312\,32123\,3132\,33322\,31}\\
$18\,22$ & \texttt{3333\,33312\,32123\,3132\,23332\,31} & $19\,22$ & \texttt{3333\,33312\,32123\,3132\,32223\,31} & $23\,24$ & \texttt{3332\,33312\,32123\,3323\,22122\,23}\\
\hline
\end{tabular}
\caption{Colourings certifying that $G''_{25}$ is critical.}
\label{TabCert2}
\end{table}

\begin{thebibliography}{99}\footnotesize\setlength{\itemsep}{0pt}

\bibitem{Abel}Z.~Abel, V.~Alvarez, E.~D.~Demaine, S.~P.~Fekete, A.~Gour,
A.~Hesterberg, P.~Keldenich, C.~Scheffer, Conflict-free coloring of graphs,
\emph{SIAM J.\ Discrete Math.}\ \textbf{32} (2018) 2675--2702.

\bibitem{BKM}S.~Bhyravarapu, S.~Kalyanasundaram, R.~Mathew, A short note on
conflict-free coloring on closed neighborhoods of bounded degree graphs,
\emph{J.\ Graph Theory} \textbf{97} (2021) 553--556.

\bibitem{Even}G.~Even, Z.~Lotker, D.~Ron, S.~Smorodinsky, Conflict-free
colorings of simple geometric regions with applications to frequency
assignment in cellular networks, \emph{SIAM J.\ Comput.}\ \textbf{33} (2003)
94--136.

\bibitem{Fabrici}I.~Fabrici, B.~Lu\v{z}ar, S.~Rindo\v{s}ov\'{a},
R.~Sot\'{a}k, Proper conflict-free and unique-maximum colorings of planar
graphs with respect to neighborhoods, \emph{Discrete Appl.\ Math.}\
\textbf{324} (2023) 80--92.

\bibitem{GJS}M.~R.~Garey, D.~S.~Johnson, L.~Stockmeyer, Some simplified
NP-complete graph problems, \emph{Theoret.\ Comput.\ Sci.}\ \textbf{1} (1976)
237--267.

\bibitem{GST}R.~Glebov, T.~Szab\'{o}, G.~Tardos, Conflict-free colouring of
graphs, \emph{Combin.\ Probab.\ Comput.}\ \textbf{23} (2014) 434--448.

\bibitem{HlinenyMalik}P.~Hlin\v{e}n\'{y}, L.~M\'{a}lik, Conflict-free
coloring planar graphs with 4 colors, in: \emph{34th Annual European Symposium
on Algorithms (ESA 2026)}, LIPIcs \textbf{388}, Schloss Dagstuhl --
Leibniz-Zentrum f\"ur Informatik, 2026, pp.~33:1--33:20.

\bibitem{PachTardos}J.~Pach, G.~Tardos, Conflict-free colourings of graphs and
hypergraphs, \emph{Combin.\ Probab.\ Comput.}\ \textbf{18} (2009) 819--834.

\bibitem{Smorodinsky}S.~Smorodinsky, Conflict-free coloring and its
applications, in: I.~B\'ar\'any, K.~J.~B\"or\"oczky, G.~Fejes T\'oth, J.~Pach
(Eds.), \emph{Geometry -- Intuitive, Discrete, and Convex}, Bolyai Soc.\ Math.\
Stud.\ \textbf{24}, Springer, 2013, pp.~331--389.

\bibitem{TsaiWest}M.-T.~Tsai, D.~B.~West, A new proof of $3$-colorability of
Eulerian triangulations, \emph{Ars Math.\ Contemp.}\ \textbf{4} (2011) 73--77.

\end{thebibliography}
\end{document}